\documentclass[12pt,oneside,openany]{article}
\usepackage{tikz}
\usepackage{tikz-cd}

\usepackage{amssymb}   
\usepackage{amsthm}
\usepackage{amsmath}
\usepackage{amsfonts}
\usepackage{latexsym}
\usepackage{graphics}
\usepackage{fancyhdr}
\usepackage{epstopdf}
\usepackage{setspace}
\usepackage{url}
\usepackage[top=1in, bottom=1in, left=.75in, right=.75in]{geometry}

\newtheorem{theorem}{Theorem}[section]

\newtheorem{prop}[theorem]{Proposition}
\newtheorem{lemma}[theorem]{Lemma}

\def\R{{\mathbb R}} 
\def\Z{{\mathbb Z}} 

\title{Generating New Partition Identities via a  Generalized Continued Fraction Algorithm}
\author{Wael Baalbaki and Thomas Garrity\\
Department of Mathematics\\
Williams College\\
Williamstown, MA 01267\\
tgarrity@williams.edu}

\begin{document}
\maketitle

\begin{abstract} Using the slow triangle map  (a type of multi-dimensional continued fraction algorithm), we exhibit a method for generating any number of new identities for subsets of integer partitions.
\end{abstract}

\section{Introduction}

Andrew and Eriksson's introduction to integer partitions \cite{Andrews-Eriksson} starts with discussing Euler's identity, ``{\it Every number has as many integer partitions into odd parts as into distinct parts}.''  As they state, this is quite surprising if you have never seen it before.  There are, though, many other equally if not more surprising partition identities.  For all there are two basic questions.  First, how to even guess the existence of any potential partition identities.  Then, once a possible potential identity is conjectured, how to prove it.

In \cite{BDGS}, Bonanno, Del Vigna, Isola and the second author developed a link between traditional continued fractions and the slow triangle map (a type of multi-dimensional continued fraction algorithm) with integer partitions of numbers into two or three distinct parts, with multiplicity.  This map was initially introduced for number theoretic reasons and has over the years exhibited interesting dynamical properties. We will see that the slow triangle map, when extended to higher dimensions, will provide a natural map $T$ from the set of integer partitions to itself.  Further, we will split the set of integer partitions into three disjoint subsets, which we will denote by $\triangle_0, \triangle_D$ and $\triangle_1$.  We will see that the triangle map $T$ will be one-to-one on $\triangle_0$  and $\triangle_1$.  Thus if we have any subset $S$ of  $\triangle_0 \cup \triangle_1$, then its image $T(S)$ will have to have the same size.  This will be used to generate many new partition identities, which appear in section four.  In section six, the corresponding generating functions are listed.

We would like to thank C. Bonanno and L. Pedersen for useful comments.

  \section{Partitions}
  
  For a general background on partition numbers, see Andrews \cite{Andrews}.  Given a positive integer $n$, the partition number $p(n)$ is the number of ways of writing $n$ as the sum of less than or equal t positive integers (ordering not mattering).  Thus $p(4)=5$ since we can write $4$ as
  $$4=3+1=2+2=2+1+1=1+1+1+1,$$
  while $p(7)= 15,$ since we can write $7$ as 
  
   $$\begin{array}{cccc}
 7& 6+1 & 5+2& 5+1+1 \\
 4+3&4+2+1& 4+1+1+1 &3+3+1  \\
3+2+2 &  3+2+1+1 &3+1+1+1+1&2+2+2+1\\
  2+2+1+1+1 & 2+1+1+1+1+1&1+1+1+1+1+1+1.
\end{array}$$
   
  We can write a given partition of $n$ as a descending sequence of positive integers
  $(\lambda_1, \ldots, \lambda_m)$ where $\lambda_\geq \lambda_2 \geq \cdots \geq \lambda_m>0$ with $n=\sum_{i=1}^m \lambda_i.$

  It is common to concatenate those $\lambda_1$ that are equal, and write the number of times a given $\lambda_i$ appears as an exponent.  Thus the above partitions of $7$ would be written as 
  
  $$\begin{array}{cccc}
  (7) & (6,1) & (5,2) & (5,1^2)  \\
   (4,3)& (4,2,1) & (4, 1^3) &(3^2,1)  \\
  (3,2^2) &  (3,2,1^2)&(3, 1^4)&(2^3,1)\\
   (2^2,1^3) & (2, 1^5)&(1^7).
\end{array}$$
Then a partition of $n$ would be some $(\lambda_1^{k_1}, \lambda_2^{k_2}, \ldots, \lambda_m^{k_m})$ where 
$\lambda_1>\lambda_2 >\cdots > \lambda_m>0$, each $k_i>0$ and $n= \sum_{i=1}^m k_i \lambda_i.$
  We call the $\lambda_i$ the {\it parts} of the partition and the $k_i$ the {\it multiplicities} of the partition.   We will always use $\lambda_i$ to denote parts and $k_i$ to denote  multiplicities. 
  
  For our purposes, we will write partitions as follows:  a partition $\lambda$ of a positive integer $n$ is of the form
  \begin{eqnarray*}
  \lambda &=& (\overline{\lambda}) \times [ \overline{k}] \\
  &=& (\lambda_1, \ldots, \lambda_m) \times [k_1, \ldots, k_m]
  \end{eqnarray*}
  where $\lambda_1>\lambda_2 >\cdots > \lambda_m>0$, each $k_i>0$ and $n= \sum_{i=1}^m k_i \lambda_i$ is the dot product of the vector of parts  $\overline{\lambda}$  with the vector of multiplicities $\overline{k}$.
  
In this notation, the partitions of  $7$ are
  
  $$\begin{array}{ccccc}
  (7)\times [1] & (6,1)\times [1,1] & (5,2)\times [1,1]& (5,1) \times [1,2]\\
 (4,3)\times [1,1] &(4,2,1)\times [1,1,1]& (4, 1) \times [1,3]&(3,1) \times[2,1]\\
   (3,2) \times [1,2]&  (3,2,1)\times [1,1,2] &(3, 1)\times [1,4]&(2,1) \times [3,1]\\
   (2,1)\times [2,3]& (2, 1)\times [1,5]&(1)\times [7].
\end{array}$$

Let $\mathcal{P}$ denote the set of all possible partitions 
$\lambda =(\overline{\lambda})\times [\overline{k}]$.  If $\lambda$ is a partition of the natural number $n$, meaning that 
$n= \sum_{i=1}^m k_i \lambda_i,$ we say that $\lambda $ has {\it size} $n$, and write it as 
$$\lambda \vdash n \; \mbox{and} \; |\lambda| =  \sum_{i=1}^m k_i \lambda_i = n.$$
If $\lambda =  (\lambda_1, \ldots, \lambda_m) \times [k_1, \ldots, k_m]$, we say that the {\it dimension} of $\lambda$ is $m$.  We denote the dimension as 
$\mbox{dim}(\lambda)$ or as  $\mbox{dim}((\overline{\lambda})\times [\overline{k}])$ The dimension of a partition will be quite important for this paper, unlike in most other work on partition numbers.   It is important enough that we denote 
$$\mathcal{P}_{m} = \{ (\overline{\lambda})\times [\overline{k}] \in \mathcal{P}: \mbox{dim}((\overline{\lambda})\times [\overline{k}])=m\}$$
and
$$\mathcal{P}_{\geq m} = \{ (\overline{\lambda})\times [\overline{k}] \in \mathcal{P}: \mbox{dim}((\overline{\lambda})\times [\overline{k}])\geq m\}.$$

Given a subset $S$ of $\mathcal{P}$, we define
$$p_S(n) = \#\{ \lambda \in S: \lambda\vdash n\}.$$

As an example of this rhetoric, Euler's identity is 
\begin{theorem} Let 
$$\mathcal{O} = \{ (\lambda_1, \ldots, \lambda_m) \times [k_1, \ldots, k_m]\in \mathcal{P}: \mbox{for all}\; i, \lambda_i \; \mbox{odd}\}$$
and 
$$\mathcal{D} = \{ (\lambda_1, \ldots, \lambda_m) \times [k_1, \ldots, k_m]\in \mathcal{P}: \forall i, k_i =1\}.$$
Then for all $n$,
$$p_{\mathcal{O}}(n) = p_{\mathcal{D}}(n) .$$
\end{theorem}

  \subsection{The slow triangle map}\label{triangle map}
  This section is an exposition on the slow triangle map, a type of multidimensional continued fraction. The fast version  originally appeared in \cite{Garrity1, GarrityT05}, where the concern was the number-theoretic  Hermite problem.    Messaoudi,  Nogueira, and  Schweiger \cite{SchweigerF08} showed that the fast map is ergodic.  Further dynamical properties were discovered by  Berth\'e, Steiner and Thuswaldner \cite{Berthe-Steiner-Thuswaldner}  and by Fougeron and  Skripchenko \cite{Fougeron-Skripchenko}.   Bonanno, Del Vigna and  Munday \cite {Bonanno- Del Vigna-Munday} and Bonanno and Del Vigna \cite{Bonanno-Del Vigna} recently used the $\R^3$ slow  triangle map to develop a tree structure of rational pairs in the plane.   In a recent preprint, Ito \cite{Ito} showed that the fast  map is self-dual (in section three of that paper).  These papers are all primarily motivated by questions from dynamics.  For general background in multidimensional continued fraction algorithms, see Karpenkov \cite{Karpenkov13} and Schweiger \cite{Schweiger4} .  
  
  Here we will simply give the basic definitions for the $m$ dimensional triangle map.  Let $m\geq 2$.
  Start with the cone 
  $$\triangle  =\{(x_1, x_2, \ldots , x_m) \in \R^m: x_1>x_2 >\cdots > x_m >0\}$$
  with the two subcones
  
  \begin{eqnarray*}
  \triangle_0  &=& \{(x_1, x_2, \ldots , x_m) \in \triangle: x_1<x_2 + x_m \}\\
   \triangle_1  &=& \{(x_1, x_2, \ldots , x_m) \in \triangle: x_1>x_2 + x_m \}\\
  \end{eqnarray*}
  and the lower dimensional cone 
  $$\triangle_D  = \{(x_1, x_2, \ldots , x_m) \in \triangle: x_1=x_2 + x_m \}.$$
  The subscript $D$ is for ``diagonal.'' 
  
  When $m=2$,  the $x_2+x_m$ term is $2x_2.$

  The slow  triangle map $T$  is a map
  $$T:\triangle_0\cup \triangle_1 \rightarrow \triangle$$
  defined as 
  
  \begin{eqnarray*}
  T(x_1, \ldots , x_m) &=& \left\{ \begin{array}{cc} T_0(x_1, \ldots , x_m) & \mbox{if} \; x_1<x_2 + x_m \\
   T_1(x_1, \ldots , x_m) & \mbox{if} \; x_1>x_2 + x_n \end{array} \right.\\
   &=& \left\{ \begin{array}{cc} (x_2, x_3, \ldots, x_m, x_1-x_2) & \mbox{if} \; x_1<x_2 + x_n \\
   (x_1-x_m, x_2, \dots , x_m) & \mbox{if} \; x_1>x_2 + x_n \end{array} \right.
\end{eqnarray*}
Reflecting its roots in dynamical systems, the map is often written as a map $T:\triangle \rightarrow \triangle$, with the convention that any 
$(x_1, \ldots, x_m) \in \triangle_D$ is ignored, as $ \triangle_D$ is a set of measure zero.
  
  As discussed in \cite{BDGS, Garrity1}, the $\R^2$ triangle map acting on some $(x_1, x_2)$ can be interpreted as a means of finding the continued fraction expansion of the ratio $x_2/x_1$.  (This connection is not obvious if you have not seen it before.)  Also, for the expert,  the fast version is just the slow version when one concatenates the $T_1$ and then applies $T_)$.  In the $\R^2$ case, the fast version is the Gauss map while the slow version is the Farey map.

  We will throughout be concentrating on the slow version.  Thus for the rest of the paper,  the ``slow triangle map'' will simply be called the triangle map.

  \section{Linking the triangle map to partitions}
  \subsection{The basic link}
  We now link the triangle map to a map from the set of integer partitions of dimension at least two to the set of partitions.  Thus we will define a map
  $$T:\mathcal{P}_{\geq 2} \rightarrow \mathcal{P}.$$
  Recall that we are denoting the set of all partition as $\mathcal{P} $ and write each partition as 
  $\lambda = (\overline{\lambda}) \times [ \overline{k}] = (\lambda_1, \ldots, \lambda_m) \times [k_1, \ldots, k_m] $
  where $\lambda_1>\lambda_2 >\cdots > \lambda_m>0$ and  each $k_i>0.$ 
  
  By a slight abuse of notation, we set 
   \begin{eqnarray*}
  \triangle_0  &=& \{(\overline{\lambda}) \times [ \overline{k}]  \in \mathcal{P}_{\geq 2}: \lambda_1<  \lambda_2 + \lambda_m \}\\
   \triangle_1  &=& \{ (\overline{\lambda}) \times [ \overline{k}]  \in\mathcal{P}_{\geq 2}: \lambda_1>  \lambda_2 + \lambda_m \}\\
 \triangle_D  &=& \{   (\overline{\lambda}) \times [ \overline{k}]  \in\mathcal{P}_{\geq 2}: \lambda_1=  \lambda_2 + \lambda_m \}
 \end{eqnarray*}
 (As before, when $m=2$,  then  $ \lambda_2 + \lambda_m$ is $2 \lambda_2.$)
 Now to define 
 $T:\mathcal{P}_{\geq 2}\rightarrow \mathcal{P}.$
 This map, this almost internal symmetry,  will allow us in the next section to find many new partition identities.
  
  Define
  
  \begin{eqnarray*}
  T((\overline{\lambda}) \times [ \overline{k}]  ) &=& \left\{ \begin{array}{cc} T_0((\overline{\lambda}) \times [ \overline{k}]  ) & \mbox{if} \; (\overline{\lambda}) \times [ \overline{k}]  \in  \triangle_0 \\
   T_1((\overline{\lambda}) \times [ \overline{k}]  ) & \mbox{if} \; (\overline{\lambda}) \times [ \overline{k}]  \in  \triangle_1    \\
   T_D((\overline{\lambda}) \times [ \overline{k}]  ) & \mbox{if} \; (\overline{\lambda}) \times [ \overline{k}]  \in  \triangle_D\end{array} \right.\\
   &=& \left\{ \begin{array}{cc} (\lambda_2, \ldots, \lambda_m, \lambda_1-\lambda_2) \times [k_1+k_2, k_3, \ldots , k_m, k_1]& \mbox{if} \; (\overline{\lambda}) \times [ \overline{k}]  \in  \triangle_0 \\
   (\lambda_1-\lambda_m, \lambda_2, \ldots, \lambda_m) \times [k_1,k_2, \ldots , k_{m-1}, k_1 + k_m]& \mbox{if} \; (\overline{\lambda}) \times [ \overline{k}]  \in  \triangle_1\\ 
    ( \lambda_2, \ldots, \lambda_m) \times [k_1+k_2,k_2, \ldots , k_{m-1}, k_1 + k_m]& \mbox{if} \; (\overline{\lambda}) \times [ \overline{k}]  \in  \triangle_D\cap \mathcal{P}_{\geq 3} \\
    (\lambda_2)\times [2k_1+k_2]  & \mbox{if} \; \begin{array}{c} (\overline{\lambda}) \times [ \overline{k}] \; \mbox{is} \\
     (2\lambda_2, \lambda_2) \times [k_1, k_2]\in  \triangle_D\end{array} \end{array} \right.
\end{eqnarray*}
  For example,
  
  \begin{eqnarray*}
  T((6, 5, 4,2) \times[k_1,k_2, k_3, k_4] &=& ( 5,4,2, 6-5)\times [k_1+k_2, k_3, k_4, k_1] \\
  &=& ( 5,4,2, 1)\times [k_1+k_2, k_3, k_4, k_1] \\
   T((9, 5, 4,2) \times[k_1,k_2, k_3, k_4] &=& ( 9-2,5, 4,2)\times [k_1,k_2, k_3, k_1+k_4] \\
  &=& (7, 5,4,2)\times [k_1,k_2, k_3, k_1+k_4] \\
T((6, 5, 4,1) \times[k_1,k_2, k_3, k_4] &=& ( 5,4,1 )\times [k_1+k_2, k_3,  k_1+ k_4] \\
T(6, 3) \times [k_1, k_2] &=& (3) \times [2k_1+k_2].
  \end{eqnarray*}
  
A straightforward calculation gives us the key to this paper:
  
  \begin{theorem} For any partition $(\overline{\lambda}) \times [ \overline{k}]  $  of dimension at least two we have 
  $$(\overline{\lambda}) \times [ \overline{k}]   \vdash n \;\mbox{if and only if}\; \;T((\overline{\lambda}) \times [ \overline{k}]  ) \vdash n.$$
  \end{theorem}

We can see 

\begin{prop} For all   $\lambda \in \triangle_0$ and for all $\lambda \in \triangle_1,$ we have 
$$\mbox{dim}(\lambda) = \mbox{dim}(T(\lambda)).$$
For all $\lambda \in \triangle_D$, the dimension under the map $T$ is reduced by one:
$$\mbox{dim}(\lambda) = \mbox{dim}(T(\lambda))+ 1.$$
  
\end{prop}

  The three subsets $\triangle_0, \triangle_1$ and $\triangle_D$ are all defined in terms of the parts $\lambda_i$ of the partition.  There are also natural subsets of $\mathcal{P}_{\geq 2}$ defined in terms of the multiplicities $k_i$, such as 
  
  \begin{eqnarray*}
M_0  &=& \{(\overline{\lambda}) \times [ \overline{k}]  \in \mathcal{P}: k_1>k_m \}\\
M_1  &=& \{ (\overline{\lambda}) \times [ \overline{k}]  \in \mathcal{P}:k_1< k_m \}\\
 \end{eqnarray*}
 
These sets correspond to $\triangle_0 $ and $\triangle_1$, as seen in:

\begin{theorem}\label{T0 and T1 1-1} The map
$$T_0:\triangle_0 \rightarrow M_0$$
and the map 
$$T_1:\triangle_1 \rightarrow M_1$$
are each one-to-one and onto maps.
\end{theorem}

\begin{proof}This can be shown simply by explicitly finding the explicit inverses to $T_0$ and to $T_1$:

  \begin{eqnarray*}
  T^{-1}((\overline{\lambda}) \times [ \overline{k}]  ) &=& \left\{ \begin{array}{cc} T_0^{-1}((\overline{\lambda}) \times [ \overline{k}]  ) & \mbox{if} \; (\overline{\lambda}) \times [ \overline{k}]  \in  M_0 \\
   T_1^{-1}((\overline{\lambda}) \times [ \overline{k}]  ) & \mbox{if} \; (\overline{\lambda}) \times [ \overline{k}]  \in  M_1    \\
  \end{array} \right.\\
   &=& \left\{ \begin{array}{cc} (\lambda_1 +\lambda_m, \lambda_1\ldots, \lambda_{m-1}) \times [k_m, k_1-k_m,k_2 \ldots , k_{m-1}]& \mbox{if} \; (\overline{\lambda}) \times [ \overline{k}]  \in  M_0 \\
   (\lambda_1+\lambda_m, \lambda_2, \ldots, \lambda_m) \times [k_1,k_2, \ldots , k_{m-1}, k_m - k_1]& \mbox{if} \; (\overline{\lambda}) \times [ \overline{k}]  \in  \triangle_1\\ 
    \end{array} \right.
\end{eqnarray*}
We are using throughout that the parts are strictly decreasing and that the multiplicities are all positive. 

\end{proof}
The map $T$ acts on both the parts $\lambda_i$ and on  the multiplicities $k_i$.  The way both $T_0$ and $T_1$ act on the parts comes 
 from the triangle map as described in Section \ref{triangle map}.  The action of each on the multiplicities is of course determined by the desire for the size of the partition to not be changed.  But once written down, it is apparent that these actions correspond to what is called in dynamical systems  ``the natural extension of a map.''  This is described in \cite{BDGS}.  For background on natural extensions, see 
Arnoux and Nogueira \cite{Arnoux-Nogueira-93}.

  \subsection{The curious case of the diagonal map $T_D$}
  The above map $T_D$ is a new map.  It is not natural if we are concerned only with the dynamics of the original triangle map.  In fact, 
in the study of the triangle map as a multidimensional continued fraction algorithm,  points on the diagonal (where $x_1=x_2+x_m$) are usually ignored, as these points form sets of measure zero.  (These ``boundary''  types points are also similarly ignored in most other multidimensional continued fraction algorithms.)  Here though we defined 
 $T_D$  on  the partitions $\lambda=  (\overline{\lambda}) \times [ \overline{k}]$ where $\lambda_1=\lambda_2+ \lambda_m$  for $m\geq 3$ and also for the quite special case of when 
 $\lambda=  (\overline{\lambda}) \times [ \overline{k}] = (2\lambda_2, \lambda_2) \times [k_1,k_2]$.   We want to see why our definition of $T_D$ is the ``right one.'' 
 
 Start with some $(\overline{\lambda}) \times [ \overline{k}]$ where $\lambda_1=\lambda_2+ \lambda_m,$  where $m\geq 3.$ In some  sense, we can think of this point as on both the boundary of  $\triangle_0$ (where $\lambda_1<\lambda_2+ \lambda_m$) and on  the boundary of  $\triangle_1$ (where $\lambda_1>\lambda_2+ \lambda_m$).  Let us start with thinking of $(\overline{\lambda}) \times [ \overline{k}]$  as actually being on $\triangle_0$ and act on it by $T_0$:
 \begin{eqnarray*}
 T_0((\overline{\lambda}) \times [ \overline{k}] )&=& (\lambda_2, \ldots, \lambda_m, \lambda_1-\lambda_2) \times [k_1+k_2, k_3, \ldots, k_m, k_1] \\
 &=& (\lambda_2, \ldots, \lambda_m, \lambda_m) \times [k_1+k_2, k_3, \ldots, k_m, k_1] .
 \end{eqnarray*}
   As the last two parts are equal, in the context of integer partitions it is natural to concatenate the last two parts, getting
   $$(\lambda_2, \ldots, \lambda_m) \times [k_1+k_2, k_3, \ldots, k_m+ k_1]  ,$$
   which is precisely how we defined the map $T_D.$

  But what if we initially think of $(\overline{\lambda}) \times [ \overline{k}]$  as  on $\triangle_1$ and  now act on it by $T_1$:
 \begin{eqnarray*}
 T_1((\overline{\lambda}) \times [ \overline{k}] )&=& (\lambda_1-\lambda_m ,\lambda_2, \ldots, \lambda_m) \times [k_1, k_2, , \ldots, k_{m-1}, k_1+k_m] \\
 &=& (\lambda_2, \lambda_2, \ldots, \lambda_m) \times [k_1, k_2, , \ldots, k_{m-1}, k_1+k_m]  .
 \end{eqnarray*}
Now the first two parts are equal.  When we concatenate these first two terms, we get 
  $$(\lambda_2, \ldots, \lambda_m) \times [k_1+k_2, k_3, \ldots, k_m+ k_1],$$
  which is once again $T_D((\overline{\lambda}) \times [ \overline{k}] .$
  
  And again, while  this concatenation is extremely natural for integer partitions, it is unnatural if one is concerned with the underlying dynamics of   the map, which is why the map $T_D$ has never been written down before.  
  
 As an example, consider 
 
 \begin{eqnarray*}
 T_0 ( 11, 8, 6, 3)  \times [k_1, k_2, k_3, k_4] &=& (8,6,3,3) \times [k_1+k_2, k_3, k_4,k_1]\\
 &\rightarrow & (8,6,3) \times [k_1+k_2, k_3, k_1+k_4]
 \end{eqnarray*}
and 
\begin{eqnarray*}
 T_1 ( 11, 8, 6, 3)  \times [k_1, k_2, k_3, k_4] &=& (8,8, 6,3) \times [k_1, k_2, k_3,k_1+ k_4]\\
 &\rightarrow & (8,6,3) \times [k_1+k_2, k_3, k_1+k_4]
 \end{eqnarray*}
  where in both cases $\rightarrow$ means to concatenate.  
  
  Unlike $T_0$ and $T_1$, this map $T_D$ is not one-to-one, as seen with 
  
  \begin{eqnarray*}
  T_D( (11, 8,6, 3) \times [2,3,4,5] &=& (8,6,3) \times [5, 3, 9]\\
  &=& T_D( (11, 8,6, 3) \times [1,4,4,6] 
  \end{eqnarray*}
  But as seen with this example, we do have the following:
  \begin{prop}\label{onetoone} Let $(\lambda_1, \ldots, \lambda_m) \times [k_1, \ldots, k_m]$ and $(\mu_1, \ldots, \mu_m) \times [l_1, \ldots, l_m]$ both be on the diagonal $\triangle_D.$
  If 
  $$T_D( (\lambda_1, \ldots, \lambda_m) \times [k_1, \ldots, k_m]) = T_D( (\mu_1, \ldots, \mu_m) \times [l_1, \ldots, l_m],$$
  then for all $i$,
  $$\lambda_i=\mu_i$$ and
  $$\lambda_1+\lambda_2= \mu_1+\mu_2, \; \lambda_1+\lambda_m= \mu_1+\mu_m.$$
  \end{prop}
  
  Thus $T_D$ is one-to-one when thought of as a map on the parts and not as a map including the multiplicities.
  
  Finally consider the case when our partition is of the form $\lambda=  (\overline{\lambda}) \times [ \overline{k}] = (2\lambda_2, \lambda_2) \times [k_1,k_2]$. 
  Thinking of $(2\lambda_2, \lambda_2) \times [k_1,k_2]$  as being on $\triangle_0$ and acting  on it by $T_0$ gives us 
 \begin{eqnarray*}
 T_0((2\lambda_2, \lambda_2) \times [k_1,k_2])&=& (\lambda_2, \lambda_2) \times [k_1+k_2,  k_1] \\
 &\rightarrow& (\lambda_2 \times [2k_1+k_2] ,
 \end{eqnarray*}
 where the $\rightarrow$ means to 
to concatenate.  Note this is how we defined $T_D$ on partitions with two parts.  

And as we would suspect, now thinking of $(2\lambda_2, \lambda_2) \times [k_1,k_2]$  as being on $\triangle_1$ and acting  on it by $T_1$ gives us

\begin{eqnarray*}
T_1((2\lambda_2, \lambda_2) \times [k_1,k_2])&=& (\lambda_2, \lambda_2) \times [k_1,  k_1+k_2] \\
 &\rightarrow& (\lambda_2 \times [2k_1+k_2] ,
 \end{eqnarray*}
 where the $\rightarrow$ still means to 
to concatenate, which gives us the same value of $T_D$.

Again, $T_D$ is natural if we want to use the triangle map to understand partitions of integers, though it is not natural if we are only interested in the underlying dynamics of the map.

  \section{New Partition Identities}\label{examples}
  We are now ready to start producing many new partition identities.

  \subsection{Applying Theorem \ref{T0 and T1 1-1}}
  First, we have 
  \begin{theorem} Every number has as many integer partitions into  partitions with $\lambda_1<\lambda_2+\lambda_m$ as into partitions with $k_1>k_m.$
  Similarly, every number has as many integer partitions into  partitions with $\lambda_1>\lambda_2+\lambda_m$ as into partitions with $k_1<k_m.$
  \end{theorem}
  
  This is just a rewriting of Theorem \ref{T0 and T1 1-1}.

    We can refine this theorem to be:

  \begin{theorem}\label{theorem 3.3 neo} Let $d$ be a positive integer.  Every number has as many integer partitions into  partitions with $\lambda_1+d = \lambda_2+\lambda_m$ as into partitions with $k_1>k_m$ and $\lambda_{m-1} =\lambda_m+d.$
  Similarly, every number has as many integer partitions into  partitions with $\lambda_1= \lambda_2+\lambda_m + d$ as into partitions with $k_1<k_m$ and $\lambda_1= \lambda_2+d.$
  \end{theorem}
  
  \begin{proof}
  Set 
  $$\triangle_0(d) = \{  (\lambda_1, \ldots, \lambda_m) \times [k_1, \ldots, k_m]\in \triangle_0: \lambda_1+d = \lambda_2+\lambda_m\}.$$
  Note that in this case, $\lambda_m=(\lambda_1-\lambda_2) + d.$
  We need to show that 
  $$T_0(\triangle_0(d)) =\{(\lambda_1, \ldots, \lambda_m) \times [k_1, \ldots, k_m]: k_1>k_m, \lambda_{m-1} = \lambda_m+d\}.$$
  
  We know that 
  \begin{eqnarray*}
  T_0( (\lambda_1, \ldots, \lambda_m) \times [k_1, \ldots, k_m] &=& (\lambda_2, \lambda_3, \ldots , \lambda_m ,\lambda_1-\lambda_2) \times [k_1+k_2 , k_3, \ldots, k_m.k_1]\\
  \end{eqnarray*}
  This means that the first multiplicity of $T_0( (\lambda_1, \ldots, \lambda_m) \times [k_1, \ldots, k_m] $ is strictly greater than the last multiplicity and that the last part is exactly $d$ less that the next to last part, which is precisely what we wanted to show.  
  
  A similar argument works for the second part of the theorem, but now setting 
  $$\triangle_1(d) = \{  (\lambda_1, \ldots, \lambda_m) \times [k_1, \ldots, k_m]\in \triangle_0: \lambda_1 = \lambda_2+\lambda_m+ d\}$$
  and then applying the one-to-one map $T_1.$

  \end{proof}
  
  \subsection{Images of Cylinders}
  We want to find interesting subsets with respect to the triangle map $T$.  It is standard in dynamics to look at the cylinder sets.  Thus the motivation for this section is coming from dynamical systems.

  Here is the idea.  The map $T_0$ is one-to-one on $\triangle_0$ and $T_1$  is one-to-one on $\triangle_1$.  Given any subset $S$ of either $\triangle_0$ or $\triangle_1$, we know that every number $n$ will have as many partitions coming from $S$ as from the appropriate $T(S).$  Thus if $S_0\subset \triangle_0$ and if $S_0\subset \triangle_1$, then for all positive integers $n$ we have 
  $$p_{S_0}(n) = p_{T_0(S_0)}(n), \; p_{S_1}(n) = p_{T_1(S_1)}(n) .$$
  
  We would like to be able to find some easily describable and natural subsets of $\triangle_0$  and $\triangle_1.$   From dynamical systems, the obvious choices would be {\it cylinder} sets.  
  
  The cylinder sets are defined as 

\begin{eqnarray*}
  \triangle_{00}  &=& \{(\overline{\lambda}) \times [ \overline{k}]  \in \triangle_0: T_0(  (\overline{\lambda}) \times [ \overline{k}]  )\in \triangle_0              \}\\
   \triangle_{01}  &=& \{(\overline{\lambda}) \times [ \overline{k}]  \in \triangle_0: T_0(  (\overline{\lambda}) \times [ \overline{k}]  )\in \triangle_1              \}\\
    \triangle_{10}  &=& \{(\overline{\lambda}) \times [ \overline{k}]  \in \triangle_1: T_1(  (\overline{\lambda}) \times [ \overline{k}]  )\in \triangle_0              \}\\ 
     \triangle_{11}  &=& \{(\overline{\lambda}) \times [ \overline{k}]  \in \triangle_1: T_1(  (\overline{\lambda}) \times [ \overline{k}]  )\in \triangle_1             \}\\ 
    \end{eqnarray*}
    
    By recursion, given a $d$-tuple $(i_1, i_2, \ldots, i_d)$ of zeros and ones, we define
     $$ \triangle_{(i_1, i_2, \ldots, i_d)}  = \{(\overline{\lambda}) \times [ \overline{k}]  \in \triangle_{(i_1}: T_{i_{1}}(  (\overline{\lambda}) \times [ \overline{k}]  )\in \triangle_{(i_2, i_3, \ldots, i_d)}.$$

  We are defining these cylinder sets in terms of the map $T$, but each can be defined more intrinsically, as follows:
      
    \begin{prop}
   \begin{eqnarray*}
   \triangle_{00} &=&  \{\overline{\lambda}) \times [ \overline{k}] \in \mathcal{P}_{\geq 3} :      \lambda_1<\lambda_2+\lambda_m, 2\lambda_2 <\lambda_1 + \lambda_3              \} \\
 && \cup  \;    \{ (\lambda_1, \lambda_2) \times [k_1,k_2]: \lambda_1 < 2\lambda_2, 3\lambda_2 < 2\lambda_1\}   \\
    \triangle_{01} &=&  \{\overline{\lambda}) \times [ \overline{k}]  :    \lambda_1<\lambda_2+\lambda_m, 2\lambda_2 >\lambda_1 + \lambda_3      \} \\
    && \cup\; \{ (\lambda_1, \lambda_2) \times [k_1,k_2]: \lambda_1 < 2\lambda_2, 3\lambda_2 >2 \lambda_1\}   \\
   \triangle_{10} &=&  \{\overline{\lambda}) \times [ \overline{k}]  :     \lambda_2+\lambda_m   <\lambda_1 < \lambda_2 + 2\lambda_m    \} \\
   && \cup \;\{ (\lambda_1, \lambda_2) \times [k_1,k_2]: \lambda_1 > 2\lambda_2, \lambda_1 <3 \lambda_2\}   \\
 \triangle_{11} &=&  \{\overline{\lambda}) \times [ \overline{k}]  :         \lambda_1>\lambda_2+2\lambda_m          \} \\
  && \cup \; \{ (\lambda_1, \lambda_2) \times [k_1,k_2]: \lambda_1 > 3\lambda_2\}   \\
   \end{eqnarray*}
   \end{prop}
   
   \begin{proof}
   
   We describe each cylinder set in turn, starting with looking at $\triangle_{00}.$  Let $\lambda  \in \triangle_{00}$  with dim($\lambda) \geq 3.$  This means that 
   $$\lambda \in \triangle_0 \; \mbox{and} \; T_0(\lambda) \in \triangle_0.$$
   To be in $\triangle_0$ means that $ \lambda_1<\lambda_2+\lambda_m$.  
   As 
   $$T_0(\lambda) = (\lambda_2, \lambda_3, \ldots, \lambda_m, \lambda_1-\lambda_2) \times [k_1+k_2, k_3, \ldots, k_m, k_1],$$
   to have $T(\lambda) \in \triangle_0$ means that we need
   $$\lambda_2< \lambda_3+ (\lambda_1-\lambda_2)$$
   which is the same as $2\lambda_2 <\lambda_1 + \lambda_3$
   
   Now let $\lambda  \in \triangle_{00}$  with dim($\lambda) =2.$  We know that for $\lambda \in \triangle_2$ in the dimension two case that $\lambda_1<2\lambda_2.$ We have 
   $$T_0(\lambda) = (\lambda_2, , \lambda_1-\lambda_2) \times [k_1+k_2, k_1],$$
    in which case to have to have $T(\lambda) \in \triangle_0$ means that 
       $$\lambda_2<2 (\lambda_1-\lambda_2)$$
   which is the same as $3\lambda_2 <\lambda_1 + \lambda_3$

   Now turn to $\triangle_{01}.$   Let $\lambda  \in \triangle_{01} \cap  \mathcal{P}_{\geq 3} .$ This means that 
   $\lambda \in \triangle_0$  and $T_0(\lambda) \in \triangle_1.$
  Being in  $\triangle_0$ still means that $ \lambda_1<\lambda_2+\lambda_m$.  
  And we still have 
   $$T_0(\lambda) = (\lambda_2, \lambda_3, \ldots, \lambda_m, \lambda_1-\lambda_2) \times [k_1+k_2, k_3, \ldots, k_m, k_1].$$
  But now we want  $T_0(\lambda) \in \triangle_1$, which means that 
   $$\lambda_2> \lambda_3+ (\lambda_1-\lambda_2)$$
   which is the same as $2\lambda_2 >\lambda_1 + \lambda_3$.
   
   The case for $\lambda  \in \triangle_{01} \cap  \mathcal{P}_{2} $ is similar.
   
   We now look at the third cylinder set $\triangle_{10}.$ Let $\lambda  \in \triangle_{10} \cap  \mathcal{P}_{\geq 3} .$ This means that 
   $\lambda \in \triangle_1$  and $T_1(\lambda) \in \triangle_0.$
  Being in  $\triangle_1$ means that $ \lambda_1>\lambda_2+\lambda_m$.  
 Then 
   $$T_1(\lambda) = (\lambda_1-\lambda_m,  \lambda_2,, \ldots, \lambda_m) \times [k_1, k_2, \ldots, k_{m-1}, k_1+k_m].$$
To have   $T_1(\lambda) \in \triangle_0$ means that 
   $$\lambda_1-\lambda_m < \lambda_2 + \lambda_m$$
   which is the same as $\lambda_1 < \lambda_2 + 2\lambda_m,$ as claimed.
   
   The case for $\lambda  \in \triangle_{10} \cap  \mathcal{P}_{2} $ is similar.

Turn to the fourth and last cylinder set $\triangle_{11}$ of the proposition. Let $\lambda  \in \triangle_{11}\cap  \mathcal{P}_{\geq 3} .$ This means that 
   $\lambda \in \triangle_1$  and $T_1(\lambda) \in \triangle_1.$
  Being in  $\triangle_1$ means that $ \lambda_1>\lambda_2+\lambda_m$, which is certainly true if $\lambda_1>\lambda_2+2\lambda_m.$ 
 As we still have 
   $T_1(\lambda) = (\lambda_1-\lambda_m ,\lambda_2,, \ldots, \lambda_m) \times [k_1, k_2, \ldots, k_{m-1}, k_1+k_m],$
  to get    $T_1(\lambda) \in \triangle_1$ will mean that 
   $$\lambda_1-\lambda_m > \lambda_2 + \lambda_m$$
   which is the same as $\lambda_1 > \lambda_2 + 2\lambda_m,$ and we are done.

And as before, the case for $\lambda  \in \triangle_{11} \cap  \mathcal{P}_{2} $ is similar.

   \end{proof}
   The above proof was done in detail just to show how straightforward it is to define explicit cylinder sets. 
   
   Then we have 
    \begin{prop}
   \begin{eqnarray*}
   T_0( \triangle_{00} )&=& \triangle_0 \cap M_0\\
   &=& \{\lambda \in \mathcal{P} : \lambda_1<\lambda_2+\lambda_m, k_1>k_m\}\\
   T_0( \triangle_{01} )&=&  \triangle_1 \cap M_0 \\
   &=& \{\lambda \in \mathcal{P} : \lambda_1>\lambda_2+\lambda_m, k_1>k_m\}\\
   T_1( \triangle_{10} )&=& \triangle_0 \cap M_1 \\
   &=& \{\lambda \in \mathcal{P} : \lambda_1<\lambda_2+\lambda_m, k_1<k_m\}   \\
   T_1( \triangle_{11} )&=&  \triangle_1\cap M_1 \\
   &=& \{\lambda \in \mathcal{P} : \lambda_1>\lambda_2+\lambda_m, k_1<k_m\}
   \end{eqnarray*}
   \end{prop}

Putting all of this together,  gives us

\begin{theorem} \label{firstcylinder}
\begin{enumerate}
\item

Every number has as many integer partitions into  partitions with   $\lambda_1<\lambda_2+\lambda_m$ and $ 2\lambda_2 <\lambda_1 + \lambda_3$   as into partitions with $\lambda_1<\lambda_2+\lambda_m$ and $k_1>k_m,$
i.e.
  $$p_{\triangle_{00}}(n) = p_{T_0(\triangle_{00})}(n) .$$

  \item

Every number has as many integer partitions into  partitions with   $\lambda_1<\lambda_2+\lambda_m$ and $ 2\lambda_2 >\lambda_1 + \lambda_3$   as into partitions with $\lambda_1>\lambda_2+\lambda_m$ and $k_1>k_m,$
 i.e.
  $$p_{\triangle_{01}}(n) = p_{T_0(\triangle_{01})}(n) .$$

  \item

Every number has as many integer partitions into  partitions with   $      \lambda_2+\lambda_m   <\lambda_1 < \lambda_2 + 2\lambda_m      $    as into partitions with  $\lambda_1<\lambda_2+\lambda_m$ and $k_1<k_m.$
   i.e.
  $$p_{\triangle_{10}}(n) = p_{T_1(\triangle_{10})}(n) .$$

    \item

Every number has as many integer partitions into  partitions with   $      \lambda_1 > \lambda_2 + 2\lambda_m      $    as into partitions with  $\lambda_1>\lambda_2+\lambda_m$ and $k_1<k_m.$
   i.e.
  $$p_{\triangle_{11}}(n) = p_{T_1(\triangle_{11})}(n) .$$

  \end{enumerate}
  \end{theorem}

We can continue applying $T$ to our sets, getting more and more new partition identities.

We have

 \begin{prop}
   \begin{eqnarray*}
   T_0(T_0( \triangle_{00} ))&=& \{\lambda \in \mathcal{P}_{\geq 3} :       k_{m-1}<k_m<k_1                       \}\\
   && \cup  \{\lambda \in \mathcal{P}_{2} :       2k_2>k_1>k_2                  \}   \\
   T_1(T_0( \triangle_{01} ))&=& \{\lambda \in \mathcal{P} :           k_1<       k_m<2k_1            \}\\
   && \cup  \{\lambda \in \mathcal{P}_{2} :     k_1<k_2<2k_1              \}   \\
   T_0(T_1( \triangle_{10} ))&=& \{\lambda \in \mathcal{P} :           k_m<    k_{m-1}, k_m <k_1                   \}\\
    && \cup  \{\lambda \in \mathcal{P}_{2} :      k_1> 2k_2        \}   \\
   T_1(T_1( \triangle_{11} ))&=& \{\lambda \in \mathcal{P} :               2k_1<k_m               \}\\   \end{eqnarray*}
   \end{prop}
Thus the set $T_i(T_j(\triangle_{ij}))$ has an intrinsic definition in terms of the parts and the multiplicities and not in terms of the maps $T_0$ and $T_1.$

\begin{proof}  These are calculations.   

We start with setting $\mu = (\mu_1, \ldots , \mu_m) \times [l_1, \ldots, l_m] \in \triangle_{00}\cap  \mathcal{P}_{\geq 3} .$ Set 
$$T_0(T_0(\mu)) = \lambda =  (\lambda_1, \ldots , \lambda_m) \times [k_1, \ldots, k_m]. $$
Then we have 
\begin{eqnarray*}
T_0(T_0(\mu)) &=&T_0(    \mu_2, \mu_3, \ldots , \mu_m, \mu_1-\mu_2) \times [l_1+l_2, l_3, \ldots, l_m, l_1]                 ) \\
&=& (\mu_3, \ldots, \mu_m, \mu_1-\mu_2, \mu_2-\mu_3) \times [l_1+l_2+l_3, l_4, \ldots , l_m, l_1, l_1+l_2]    \\
&=&  (\lambda_1, \ldots , \lambda_m) \times [k_1, \ldots, k_m].
\end{eqnarray*}
We have 
$$k_{m-1} = l_1<l_1+l_2 =k_m$$
and
$$k_m = l_1+l_2 < l_1+l_2+l_3 = k_1,$$
and we have our description for $T_0(T_0( \triangle_{00} \cap  \mathcal{P}_{\geq 3}  ))$.

Now let $\mu = (\mu_1,  \mu_2) \times [l_1,  l_2] \in \triangle_{00}\cap  \mathcal{P}_{2} .$ Following along the lines of the above, set 
$$T_0(T_0(\mu)) = \lambda =  (\lambda_1, \lambda_2) \times [k_1, k_2]. $$

Then 
\begin{eqnarray*}
T_0(T_0  (\mu)) &=&T_0(    \mu_2, \mu_1-\mu_2) \times [l_1+l_2, l_1]                 ) \\
&=& (\mu_1-\mu_2,       2  \mu_2-\mu_1) \times [2l_1+l_2,    , l_1+l_2]    \\
&=&  (\lambda_1, \lambda_2) \times [k_1,  k_2].
\end{eqnarray*}
Then we have 

$$k_2 = l_1+l_2 < 2l_1+l_2=k_1 <  2l_1+2l_2=2k_2$$
as desired.

Now let $\mu = (\mu_1, \ldots , \mu_m) \times [l_1, \ldots, l_m] \in \triangle_{01} \cap  \mathcal{P}_{\geq 3} .$
We have 
\begin{eqnarray*}
T_1(T_0   (\mu)) &=&T_0(    \mu_2, \mu_3, \ldots , \mu_m, \mu_1-\mu_2) \times [l_1+l_2, l_3, \ldots, l_m, l_1]                 ) \\
&=& (2\mu_2-\mu_1, \mu_3, \ldots,  \mu_m, \mu_1-\mu_2) \times [l_1+l_2, l_3, \ldots, l_m, 2l_1+l_2 ] \\
&=&  (\lambda_1, \ldots , \lambda_m) \times [k_1, \ldots, k_m].
\end{eqnarray*}
We have 
$$k_1 = l_1+l_2<k_m = 2l_1+l_2 < 2l_1+2l_2 = 2k_1,$$
as we want.

The case when  $\mu = (\mu_1,  \mu_2) \times [l_1,  l_2] \in \triangle_{01}\cap  \mathcal{P}_{2} $ is similar.

On to the third  type of cylinder    set.  Now we let $\mu = (\mu_1, \ldots , \mu_m) \times [l_1, \ldots, l_m] \in \triangle_{10}    \cap  \mathcal{P}_{\geq 3} .$   We have 
\begin{eqnarray*}
T_0(T_1   (\mu)) &=&T_0(  \mu_1-\mu_m , \mu_2, \ldots , \mu_m) \times [l_1, l_2, , \ldots, l_{m-1}, l_1+ l_m]                 ) \\
&=&      (\mu_2, \dots, \mu_m, \mu_1-\mu_2-\mu_m) \times [l_1+l_2, l_3, \ldots, l_1+l_m, l_1    ]      \\
&=&  (\lambda_1, \ldots , \lambda_m) \times [k_1, \ldots, k_m].
\end{eqnarray*}
which gives us 
$$k_m = l_1 <l_1+l_2 =k_1, \; k_m = l_1 <l_1+l_m =k_{m-1}.$$

The case when  $\mu = (\mu_1,  \mu_2) \times [l_1,  l_2] \in \triangle_{10}\cap  \mathcal{P}_{2} $ is similar.

Now for the last condition.  
 Let $\mu = (\mu_1, \ldots , \mu_m) \times [l_1, \ldots, l_m] \in \triangle_{11}  .$  Then
\begin{eqnarray*}
T_1(T_1   (\mu)) &=&T_1( ( \mu_1-\mu_m , \mu_2, \ldots , \mu_m) \times [l_1, l_2, , \ldots, l_{m-1}, l_1+ l_m]                 ) \\
&=&     ( \mu_1-2\mu_m , \mu_2, \ldots , \mu_m) \times [l_1, l_2, , \ldots, l_{m-1}, 2l_1+ l_m]       \\
&=&  (\lambda_1, \ldots , \lambda_m) \times [k_1, \ldots, k_m].
\end{eqnarray*}
which gives us 

$$2k_1= 2l_1 < 2l_1 + l_m = k_m$$
using critically that $l_m\geq 1,$ giving us our last equality.

Thus we have 

\begin{theorem} \label{secondcylinder}
\begin{enumerate}
\item

Every number has as many integer partitions into  partitions of at least three parts with   $\lambda_1<\lambda_2+\lambda_m$ and $ 2\lambda_2 <\lambda_1 + \lambda_3$  or into partitions of exactly two parts with $2\lambda_2>\lambda_1$ and $2\lambda_1>3\lambda_2$   as into partitions  of at least three parts with $k_{m-1}<k_m<k_1$ or into partitions of exactly two parts with $2k_2>k_1>k_2.$
i.e.
  $$p_{\triangle_{00}}(n) = p_{T_0(T_0(\triangle_{00}))}(n) .$$

  \item

Every number has as many integer partitions into  partitions  of at least three parts  with   $\lambda_1<\lambda_2+\lambda_m$ and $ 2\lambda_2 >\lambda_1 + \lambda_3$  or into partitions of exactly two parts with                  
 $2\lambda_2>\lambda_1$ and $2\lambda_1<3\lambda_2$     as into partitions    with $k_1<k_m < 2k_1,$
 i.e.
  $$p_{\triangle_{01}}(n) = p_{T_1(T_0(\triangle_{01}))}(n) .$$

  \item

Every number has as many integer partitions into  partitions  of at least three parts  with   $      \lambda_2+\lambda_m   <\lambda_1 < \lambda_2 + 2\lambda_m      $   or into partitions of exactly two parts with 
$2\lambda_2<\lambda_1< 3\lambda_2$  as into partitions of at least three parts with  $k_m<k_{m-1}$ and $k_m<k_1$  or into partitions of exactly two parts with $2k_2<k_1.$
   i.e.
  $$p_{\triangle_{10}}(n) = p_{T_0(T_1(\triangle_{10})}(n) .$$

    \item

Every number has as many integer partitions into  partitions with   $      \lambda_1 > \lambda_2 + 2\lambda_m      $    as into partitions with  $2k_1<k_m.$
   i.e.
  $$p_{\triangle_{11}}(n) = p_{T(T_1(\triangle_{11})}(n) .$$

  \end{enumerate}
  \end{theorem}

\end{proof}

Let us work out the example of $\triangle_{01}(11)$.   In the appendix we list all partitions of the number $11$.  
We see that

$$\triangle_{01}=  \{ (6,5)   \times [1,1] ,    (5,4,2)   \times [1,1,1]   ,    (4,3)   \times [2,1]     \}$$
As each partition is in $\triangle_0$, to apply the map $T$ we must apply $T_0$:
       
\begin{eqnarray*}
(6,5)   \times [1,1]        &   \xrightarrow{T_0   } &   (5,1) \times [2,1]     \\ 
  (5,4,2)   \times [1,1,1]      &   \xrightarrow{T_0   } &  (4,2,1) \times [2,1,1]      \\ 
     (4,3)   \times [2,1]      &   \xrightarrow{T_0   } &  (3,1) \times [3,2]      \\ 
       \end{eqnarray*}
Then we get 
$$T(\triangle_{01})= \{  (5,1) \times [2,1] ,  (4,2,1) \times [2,1,1] , (3,1) \times [3,2] \},$$
as desired.  Each of these partitions is in $\triangle_1$, and hence we must apply $T_1$:
\begin{eqnarray*}
   (5,1) \times [2,1]     &   \xrightarrow{T_1   } & (4,1) \times [2,3]\\
   (4,2,1) \times [2,1,1]     &   \xrightarrow{T_1   } & (3,2,1) \times [2,1,3]  \\
    (3,1) \times [3,2]      &   \xrightarrow{T_1   } & (2,1) \times [3,5]
       \end{eqnarray*}
Thus 
by calculation we have 
$$T(T(\triangle_{01})) = \{ (4,1) \times [2,3],  (3,2,1) \times [2,1,3] ,  (2,1) \times [3,5]\}.$$

 Note that a description of a cylinder set $\triangle_{(i_1, i_2, \ldots, i_d)} $ is solely in terms of the parts $n_1$ while the description of its image under $d$ iterations of $T$ is solely in terms of the multiplicities.  It strikes us that it would be enjoyable to find formulas for each of these sets in a straightforward manner.

 There are as many partition identities as there are cylinder sets.  
 
 We will just look at one example.
 
 Set 
 $$\triangle_d^G =\{   (\lambda_1, \ldots , \lambda_m) \times [k_1, \ldots, k_m]\in \mathcal{P}: \lambda_1 -\lambda_2-d\lambda_m>0>   \lambda_1 -\lambda_2-(d+1)\lambda_m \}.$$
 The $G$ stands for Gauss.  When $m=2,$ these sets are the natural domains for the traditional Gauss map for continued fractions. The set $\triangle_d^G $ is the cylinder set for $d$ ones followed by a zero and for $m=3$ are the cylinder sets in \cite{Garrity1}.

 \begin{lemma} For all $0\leq n \leq d$, the map  $T_1^n$ is a one-to-one onto function from $\triangle_d^G$ to 
 
 $$ \triangle_{d-n}^G \cap \{(\lambda_1, \ldots , \lambda_m) \times [k_1, \ldots, k_m]: nk_1<k_m\}.$$
 \end{lemma}
   
 \begin{proof}
 
 Similar to before, set $\mu = (\mu_1, \ldots , \mu_m) \times [l_1, \ldots, l_m] \in \triangle_d^G.$ Set 
$$T_1^n( \mu) =  (\lambda_1, \ldots , \lambda_m) \times [k_1, \ldots, k_m]. $$

We have 
\begin{eqnarray*}
T_1^n( \mu) &=& T_1^n( (\mu_1, \ldots , \mu_m) \times [l_1, \ldots, l_m] )\\
&=& (\mu_1-n\mu_m, \mu_2, \ldots, \mu_m)\times [l_1, \ldots, nl_1+ l_m] \\
&=&  (\lambda_1, \ldots , \lambda_m) \times [k_1, \ldots, k_m].
\end{eqnarray*}
We have 
\begin{eqnarray*}
\lambda_1- \lambda_2 - (d-n)\lambda_m &=& \mu_1-\mu_2 - d\mu_m \\
&>& 0\\
&>& \mu_1-\mu_2 - d\mu_m - \mu_m\\
&=& \lambda_1- \lambda_2 - (d-n+1)\lambda_m
\end{eqnarray*}
For the multiplicities, we have 
\begin{eqnarray*}
n k_1 &=& n l_1 \\
&< & n l_1 + l_m \\
&=& k_m
\end{eqnarray*}

The one-to-oneness  comes from that $T_1$ is one-to-one.  It can be checked that each of the maps has a well-defined inverse, giving us that the iterative map $T_1^n$ must also be onto.
 
 \end{proof}
 This give us in particular that the image of $ \triangle_d^G$ under $T_1^d$ is in $ \triangle_0^G= \triangle_0.$
 
 \begin{lemma}The map $T_0\circ T_1^d$ is a one-to-one onto function from $\triangle_d^G$ to 
 $$\{(\lambda_1, \ldots , \lambda_m) \times [k_1, \ldots, k_m]: dk_m<k_{m-1}, k_m < k_1\}.$$
 \end{lemma}
 
 \begin{proof} This will be similar to the previous lemma.  Set $\mu = (\mu_1, \ldots , \mu_m) \times [l_1, \ldots, l_m] \in \triangle_d^G.$ Set 
$$T_0\circ T_1^d( \mu) =  (\lambda_1, \ldots , \lambda_m) \times [k_1, \ldots, k_m]. $$
Then
\begin{eqnarray*}
T_0\circ T_1^d( \mu) &=& T_0 \circ T_1^d( (\mu_1, \ldots , \mu_m) \times [l_1, \ldots, l_m] )\\
&=&T_0( (\mu_1-n\mu_m, \mu_2, \ldots, \mu_m)\times [l_1, \ldots, dl_1+ l_m] )\\
&=& ( \mu_2, \ldots, \mu_m, \mu_1-\mu_2 - d\mu_m) \times [l_1+l_2, l_3 \ldots, dl_1+l_m, l_1]\\
&=&  (\lambda_1, \ldots , \lambda_m) \times [k_1, \ldots, k_m].
\end{eqnarray*}
which gives us our result.  

And as before, the one-to-oneness  comes from that $T_1$ and $T_0$ are one-to-one.  And it can again  be checked that each of the maps has a well-defined inverse, giving us that the iterative map $T_0 \circ T_1^d$ must also be onto.
 \end{proof}
 
 These two lemmas give us the following partition identity.

 \begin{theorem}  Every number has as many integer partitions into  partitions with  $  \lambda_1 -\lambda_2-d\lambda_m>0>   \lambda_1 -\lambda_2-(d+1)\lambda_m $  as into partitions, for any $0\leq p\leq d,$  $  \lambda_1 -\lambda_2-(d-p)\lambda_m>0>   \lambda_1 -\lambda_2-(d-p+1)\lambda_m $  and $ pk_1<k_m$, and as many partitions with $dk_m<k_{m-1}, k_m < k_1$, i.e.
  $$p_{\triangle_{d}^G}(n) = p_{T_1^p(\triangle_{d}^G)}(n)  = p_{T_0 \circ T_1^p(\triangle_{d}^G)}(n) .$$
 \end{theorem}

  \subsection{Extending distinctness of parts via $T$}
  The map $T$ from partitions to partitions allows us to start with    any already known partition identity, apply the map $T$ and see what happens.
  
  We will do this for the set $\mathcal{D}$ of partitions with distinct parts (where all the $k_i=1$) in this subsection and then for the set $\mathcal{O}$ of partitions with all the parts are odd (where each $\lambda_i$ is an odd number) in the next subsection.
  
  We start with

  \begin{theorem}\label{distinct}   For any positive number $n$, the  number of integer partitions into  partitions with  distinct parts (with $k_i=1$ for all $i$)  is precisely equal to $1$ plus the number of partitions  into partitions all of whose parts are distinct, save for the largest part which will have multiplicity two ($k_1=2$ and $k_i=1$ for $i >1$)  plus  those partitions all of whose parts are distinct, save for the smallest part which will have multiplicity two ($k_m=2$ and $k_i=1$ for $i < m$) plus those partitions all of whose parts are distinct, save for the smallest and the largest parts, each having multiplicity two ($k_1=k_m=2$ and $k_i=1$ for $1<i<m$), plus an additional $1$ if $n$ is divisible by $3$.

  \end{theorem}
   In other words, recalling that $\mathcal{D}$ is all partitions with distinct parts and setting 
   \begin{eqnarray*}
   \mathcal{E}_0 &=&\{ (\lambda_1, \ldots, \lambda_m) \times [k_1, \ldots, k_m]\in \mathcal{P}_{\geq 2}: k_1=2, k_i =1 \; \mbox{if} \; i\geq 2\}\\
    \mathcal{E}_1 &=&\{ (\lambda_1, \ldots, \lambda_m) \times [k_1, \ldots, k_m]\in \mathcal{P}_{\geq 2}: k_m=2, k_i =1 \; \mbox{if} \; i<m\}\\
     \mathcal{E}_D &=&\{ (\lambda_1, \ldots, \lambda_m) \times [k_1, \ldots, k_m]\in \mathcal{P}_{\geq 2}: k_1=k_m=2, k_i =1 \; \mbox{if} \; 1<i<m\},\\
   \end{eqnarray*}
then the theorem is stating that 
$$p_{\mathcal{D}} (n) = 1 + p_{\mathcal{E}_0 } (n) +p_{\mathcal{E}_1 } (n) +p_{\mathcal{E}_D } (n)  + \left\lfloor    (3/n)\lfloor 3/n \rfloor   \right\rfloor,$$
where the $\left\lfloor    (3/n)\lfloor 3/n \rfloor   \right\rfloor$ term is giving us the additional $1$ when $n$ is divisible by $3$
   
  \begin{proof}  
 
  We know that 
  $$\mathcal{D} = \{ (n)\times [1]: n=1,2,3,\ldots \} \cup (\mathcal{D} \cap \triangle_0) \cup  (\mathcal{D} \cap \triangle_1) \cup  (\mathcal{D} \cap \triangle_D) .$$
   The first $1$ in the desired equation  is to account for the $ \{ (n)\times [1]\} $ term.  
  
   We will first show that $T_0$ is a one-to-one onto map from  $\mathcal{D} \cap \triangle_0$    to $ \mathcal{E}_0 $. 
   Let 
   $$\lambda =(\lambda_1, \ldots, \lambda_m) \times [1\ldots, 1].$$
   Then $T_0(\lambda)$ is 
   $$(\lambda_2, \ldots, \lambda_m, \lambda_1-\lambda_2] \times[2,1, \ldots, 1]$$
   as desired.  To show ontoness, we simply have to start with an element of the form 
  $$(\lambda_1, \ldots, \lambda_m, ] \times[2,1, \ldots, 1]$$
   and apply $T_0^{-1}$, getting 
   $$(\lambda_1+\lambda_m, \lambda_1, \ldots, \lambda_{m-1}) \times [1, \ldots, 1]$$
   which is indeed in $\mathcal{D} \cap \triangle_0$ 
   
   The argument that $T_1$ is an onto map from  $\mathcal{D} \cap \triangle_1$    to $ \mathcal{E}_1 $ is similar.
   
  Both of these use that $T_0$ and $T_1$ are always one-to-one maps.
  
Now to show that   $T_D$  is a one-to-one onto map from $\mathcal{D} \cap \triangle_D \cap \mathcal{P}_{\geq 3}$  to  $\mathcal{E}_D $.  In general, $T_D$ is not one-to-one, but it will be in this case.
  
  Let us start with $\lambda =(\lambda_1, \ldots, \lambda_m) \times [1\ldots, 1] \in \mathcal{D} \cap \triangle_D.$
  Then $T_D(\lambda)$  is
  $$(\lambda_2, \ldots, \lambda_m) \times [2, 1 \ldots, 1,2].$$
  One-to-oneness follows from Proposition \ref{onetoone} , as as does ontoness.  
  
  Finally, turn to  $T_D$   acting on $\mathcal{D} \cap \triangle_D \cap \mathcal{P}_{2}$.  An element in this set has the form $(2\lambda, \lambda) \times [1,1],$ and hence only deals with partitions of numbers that are divisible by $3$.  By definition $T_D(  (2\lambda, \lambda) \times [1,1]  ) = (\lambda)\times [3].$  This is what gives us the needed additional $1$ is $n$ if divisible by $3$, and hence the final term of $\left\lfloor    (3/n)\lfloor 3/n \rfloor   \right\rfloor$.   
  
  \end{proof}

  For an example, we will see how $T$ is   a one-to-one onto map from $  \mathcal{D}(11)\cap \triangle_0$ to $\mathcal{E}_0(11)$.
  From the appendix, we have
  $$   \mathcal{D}(11)\cap \triangle_0 =  \{ (7, 4)   \times [1, 1]  ,  (6,5)   \times [1,1]  ,  ( (5,4,2)   \times [1,1,1]   \}$$
  and
  $$\mathcal{E}_0(11) = \{   (4,3) \times [2,1]  , (4,3) \times [2,1]  ,   (4,2,1) \times [2,1,1]   \}.$$
  Applying $T$:
  \begin{eqnarray*}
    (7, 4)   \times [1, 1] &   \xrightarrow{T_0   }   &    (4,3) \times [2,1]      \\
      (6,5)   \times [1,1]   &   \xrightarrow{T_0   }   &   (5,1) \times [2,1]       \\ 
           ( (5,4,2)   \times [1,1,1]   &  \xrightarrow{T_0   }    &     (4,2,1) \times [2,1,1]     \\ 
             \end{eqnarray*}
giving us our result.

  \subsection{Extending oddness  of parts  via $T$}
  We want to show

  \begin{theorem}\label{odd and T} For any positive number $n$, the  number of integer partitions into  partitions with  only odd parts  is precisely equal to  the number of odd factors of $n$ plus  the number of  partitions all of whose parts are all odd, save the smallest part which must be even and with the multiplicity of the largest part strictly greater than the multiplicity of the smallest part plus the number of partitions all of whose parts are all odd, save the  largest  part which must be even and with the multiplicity of the largest part strictly smaller than the multiplicity of the smallest part.

  \end{theorem}
  
  In other words, using that $\mathcal{O}$ is the set of all partitions all of whose parts are odd and setting 
   \begin{eqnarray*}
   \mathcal{F}_0 &=&\{ (\lambda_1, \ldots, \lambda_m) \times [k_1, \ldots, k_m]\in \mathcal{P}_{\geq 2}: \lambda_m \;\mbox{even}, \lambda_i  \;\mbox{odd}\; \mbox{if}\; i<m, k_1>k_m\}\\
  \mathcal{F}_1 &=&\{ (\lambda_1, \ldots, \lambda_m) \times [k_1, \ldots, k_m]\in \mathcal{P}_{\geq 2}: \lambda_1 \;\mbox{even}, \lambda_i  \;\mbox{odd}\; \mbox{if}\; i>1, k_1<k_m\}\\
     \end{eqnarray*}
then the theorem is stating that 
$$p_{\mathcal{O}} (n) = (\mbox{number of odd factors of}\; n) + p_{\mathcal{F}_0 } (n) +p_{\mathcal{F}_1 } (n) .$$

  \begin{proof} By now the path is clear.  First, if $2k+1$ divides $n$, then one of the odd partitions would be 
  $(2k+1) \times [n/(2k+1)]$.  Each of these has dimension one, and hence the triangle map would not help us with this particular odd partition.  This is why we need to explicitly add in the ``number of odd factors'' of $n$.

 Next, we  split the partitions of odd parts $\mathcal{O}$  of dimension at least two into the three disjoint sets 
  $$\mathcal{O} \cap \mathcal{P}_{\geq 2}= (\mathcal{O} \cap \triangle_0) \cup  (\mathcal{O} \cap \triangle_1) \cup  (\mathcal{O} \cap \triangle_D) $$
  and find a clean description of the image of each of these sets under the map $T$.  In this case, the last set $\mathcal{O} \cap \triangle_D$ is empty, since we cannot have 
  $\lambda_1=\lambda_2+\lambda_m$ with all three of $\lambda_1, \lambda_2, \lambda_m$ being odd numbers.
  
  Let $\lambda\in \mathcal{O} \cap \triangle_0.$ Then 
  \begin{eqnarray*}
  T_0(\lambda) &=& T_0((\lambda_1, \ldots, \lambda_m) \times [k_1, \ldots, k_m]) \\
 &=& (\lambda_2, \ldots, \lambda_m, \lambda_1-\lambda_2] \times[k_1+k_2,k_3, \ldots, k_m, k_1,]
  \end{eqnarray*}
  a partition all of whose parts are odd, save the smallest, which is indeed even.  Further, the multiplicity of the largest part $k_1+k_2$, is clearly strictly greater than the multiplicity of the smallest part, $k_1$.
  We know that $T_0$ is one-to-one.  We need to show that if we start with a partition $\lambda$ all of whose parts are odd, save the smallest, that $T_0^{-1}$ is in $\mathcal{O}.$
  
  Since the multiplicity of the largest part is strictly greater than the multiplicity of the smallest part, we know that we must act on the partition by $T_0^{-1}:$
  
  \begin{eqnarray*}
  T_0^{-1}(\lambda) &=& T_0^{-1}((\lambda_1, \ldots, \lambda_m) \times [k_1, \ldots, k_m]) \\
 &=& (\lambda_1+\lambda_m, \lambda_1,\lambda_2, \ldots, \lambda_{m-1}) \times[k_m, k_1-k_2, k_2, k+3, \ldots, k_{m-1}]
  \end{eqnarray*}
  all of whose parts are odd.

  The argument that $T_1$ is a one-to-one and onto map from  $\mathcal{O} \cap \triangle_1$  to the set of all partitions all of whose parts are odd, save the  largest  part which must be even and with the multiplicity of the largest part strictly smaller than the multiplicity of the smallest part, is similar.

  \end{proof}

  Putting the two theorems from the last two subsections together, along with Euler's original identity, we have 
  
  \begin{theorem} The number of ways of partitioning any positive integer $n$ from each of the following three sets are all equal. 
  \begin{enumerate}
  \item All parts are distinct
  \item All parts are odd
  \item $\{(n) \times [1]\}$ union  the set where all parts are odd, save the smallest part which must be even and with the multiplicity of the largest part strictly greater than the multiplicity of the smallest part, union the set where  all  parts are all odd, save the  largest  part which must be even and with the multiplicity of the largest part strictly smaller than the multiplicity of the smallest part.
  \end{enumerate}
  Thus 
  \begin{eqnarray*}
  p_{\mathcal{D}} (n) &=&p_{\mathcal{O}} (n) \\
  &=&1 + p_{\mathcal{E}_0 } (n) +p_{\mathcal{E}_1 } (n) +p_{\mathcal{E}_D } (n) \\
   &=& (\mbox{number of odd factors of}\; n) + p_{\mathcal{F}_0 } (n) +p_{\mathcal{F}_1 } (n).
  \end{eqnarray*}
  
  \end{theorem}
  
  \section{ A Summary of Relevant Sets}\label{sets}
  
  We list in a table most of  the various sets that we have cared about in this paper.  For all, we use that the $\lambda_i$ make up the parts and the $k_i$ the multiplicities.
  For all $m$, we have 
  $\lambda_1> \cdots, > \lambda_m>0$ and $k_i>0$ for $i=1, \ldots, m.$

$$\begin{array}{c|c|c}
\mbox{sets}  & \mbox{dim}\;=2 &  \mbox{dim}\;\geq 3 \\
\hline
\triangle_0   &     2\lambda_2>\lambda_1        &   \lambda_2 + \lambda_m > \lambda_1         \\
\hline
\triangle_1   &      2\lambda_2 <\lambda_1        &   \lambda_2 + \lambda_m < \lambda_1        \\
\hline
\triangle_D   &      2\lambda_2 =\lambda_1        &   \lambda_2 + \lambda_m = \lambda_1        \\
\hline
\triangle_{00  }   & 2 \lambda_2>\lambda_1 , 2\lambda_1>3\lambda_2        &     \lambda_2 + \lambda_m > \lambda_1   , 2\lambda_2 <\lambda_1 + \lambda_3            \\
\hline
\triangle_{01  }   &   2\lambda_2>\lambda_1     , 2\lambda_1< 3\lambda_2      &        \lambda_2 + \lambda_m > \lambda_1 , 2\lambda_2 >\lambda_1 + \lambda_3        \\
\hline
\triangle_{10  }   &    2\lambda_2 <\lambda_1 , 3\lambda_2>\lambda_1    &    \lambda_2 + \lambda_m < \lambda_1    , \lambda_2+2\lambda_m >\lambda_1    \\
\hline
\triangle_{11  }   &    3\lambda_2 <\lambda_1    &   \lambda_2 +2 \lambda_m < \lambda_1        \\
\hline
M_0=T_0(\triangle_0)&     k_1>k_2        &      k_1>k_m         \\
\hline
M_1  = T_1(\triangle_1)  &    k_1<k_2       &    k_1<k_m      \\
\hline
T_{  0   }     ( \triangle_{00}  )  &      2\lambda_2>\lambda_1  , k_1>k_2      &      \lambda_2 + \lambda_m > \lambda_1       , k_1>k_m           \\
\hline
T_{  0   }     ( \triangle_{01}  )   &       2\lambda_2 <\lambda_1     , k_1>k_2             &     \lambda_2 + \lambda_m < \lambda_1     , k_1>k_m            \\
\hline
T_{  1   }     ( \triangle_{10}  )   &          2\lambda_2>\lambda_1 , k_1<k_2               &         \lambda_2 + \lambda_m > \lambda_1   , k_1<k_m         \\
\hline
T_{  1   }     ( \triangle_{11}  )   &          2\lambda_2<\lambda_1 , k_1<k_2             &      \lambda_2 + \lambda_m < \lambda_1      , k_1<k_m       \\
\hline

T_{  0   }  (  T_0 ( \triangle_{00}  ) )  &   2k_2>k_1>k_2     &    k_1>k_m>k_{m-1}         \\
\hline
T_1 (T_{  0   }     ( \triangle_{01}  )  ) &        2k_1>k_2>k_1          &    2k_1>k_m>k_1            \\
\hline
T_0(T_{  1   }     ( \triangle_{10}  ) )  &      2k_2<k_1              &   k_1>k_m, k_{m-1}>k_m        \\
\hline
T_1(   T_{  1   }     ( \triangle_{11}  )  ) &         2k_1<k_2         &       k_m>2k_1       \\
\hline
\mathcal{D} & k_1=k_2=1   & k_1=\ldots =k_m = 1\\
\hline
\mathcal{E}_0 & k_1=2, k_2=1   & k_1=2, k_2\ldots =k_m = 1\\
\hline
\mathcal{E}_1 & k_1=1, k_2=2   & k_1\ldots =k_{m-1}, k_m = 2\\
\hline
\mathcal{E}_D & k_1=2, k_2=2   & k_1=2, k_2\ldots =k_{m-1}, k_m = 2\\
\hline
\mathcal{O} & \lambda_1, \lambda_2\;\mbox{odd}   &  \lambda_i \;\mbox{odd}, i=1, \ldots, m\\
\hline
\mathcal{F}_0 & \lambda_1 \;\mbox{odd},  \lambda_2 \;\mbox{even} & \lambda_i \;\mbox{odd}\; i=1\ldots, m-1, \lambda_2 \; \mbox{even}\\
\hline
\mathcal{F}_1 & \lambda_1 \;\mbox{even},  \lambda_2 \;\mbox{odd} &\lambda_1 \; \mbox{even},  \lambda_i \;\mbox{odd}\; i=2\ldots, m\\
\hline
\end{array}$$

All of these sets only make sense for partitions whose dimensions are greater than or equal to two, save for the sets $\mathcal{D} $ and $\mathcal{O},$ as we have
\begin{eqnarray*}
\mathcal{D} \cap \mathcal{P}_1 &=& \{ (n) \times [1]: n=1,2,3,\ldots \} \\
\mathcal{O} \cap \mathcal{P}_1 &=& \{ (3n) \times [k]: n,k=1,2,3,\ldots, \} 
\end{eqnarray*}

\section{ On Their Generating Functions}
 
 It is straightforward to find  the generating functions for the sets defined in  Section \ref{sets}, as we will see. This will allow us to translate each of our partition identities into  identities of the corresponding generating functions.  

We will use the standard notation 
\begin{eqnarray*}
(a;q)_0&=& 1\\
(a;q)_n &=& (1-a)(1-aq) \cdots (1-aq^{n-1})  \\
(a;q)_{\infty} &=& (1-a)(1-aq)(1-aq^2 ) \cdots    
\end{eqnarray*}

  We started with the set $\mathcal{P}$ of all partitions $(\lambda_1, \ldots, \lambda_m)\times [k_1, \ldots, k_m]$ where $\lambda_1>\lambda_2 >\cdots >\lambda_m>0$, $k_i>0$ and for all $i$, $\lambda_i, k_i \in \Z.$
Then it is well known that the generating function for $p_{\mathcal{P}}(n) $ is 

$$\sum_{n=0}^{\infty}  p_{\mathcal{P}}(n)q^n  = \frac{1}{(q: q)_{\infty} } =     \prod_{m=1}^{\infty}  \frac{1}{(1-q^m)} .$$
Here we are using the convention that $p_{\mathcal{P}}(0) =1.$

We set 
$$\mathcal{P}_N = \{(\lambda_1, \ldots, \lambda_m)\times [k_1, \ldots, k_m]\in \mathcal{P}: m=N\},$$
all the partitions of dimension $N$ and let $\mathcal{P}_{\geq N}$ be the set of all partitions with dimension at least $N$. 
Then we have the generating function for partitions of dimension one being

 $$\sum_{n=0}^{\infty}  p_{\mathcal{P}_1}(n)q^n  = \sum_{n=1}^{\infty } d(n) q^n =  \sum_{m=1}^{\infty} \frac{q^m}{1-q^m},$$
where $d(n)$ is the divisor function of $n$, meaning that it is the number of divisors of $n$, including $1$ and the number $n$ itself.

More generally, we have 
$$\sum_{n=0}^{\infty}  p_{\mathcal{P}_N}(n)q^n  = \sum_{\lambda_1>\cdots >\lambda_N>0} \frac{q^{\lambda_1}}{(1-q^{\lambda_1}) }\cdots  \frac{q^{\lambda_N}}{(1-q^{\lambda_N}) }.$$

We know that 
$$\mathcal{P}(n) = \{ (m)\times [k]: mk=n\} \cup \triangle_{0}(n)  \cup \triangle_{1}(n)  \cup \triangle_{D}(n) $$
where all of these sets are disjoint from each other.

 We can calculate that 
 \begin{eqnarray*}
 \sum_{n=0}^{\infty}  p_{\triangle_0}(n)q^n  &=&   \sum_{m=2}^{\infty}    \sum_{\left(\begin{array}{c}\lambda_1>\cdots >\lambda_m>0\\ \lambda_1<  \lambda_2 + \lambda_m \end{array}\right)} \frac{q^{\lambda_1}}{(1-q^{\lambda_1}) }\cdots  \frac{q^{\lambda_m}}{(1-q^{\lambda_m}) }       \\
 \sum_{n=0}^{\infty}  p_{\triangle_1}(n)q^n  &=&          \sum_{m=2}^{\infty}    \sum_{\left(\begin{array}{c}\lambda_1>\cdots >\lambda_m>0\\ \lambda_1>  \lambda_2 + \lambda_m \end{array}\right)} \frac{q^{\lambda_1}}{(1-q^{\lambda_1}) }\cdots  \frac{q^{\lambda_m}}{(1-q^{\lambda_m}) }      \\
\sum_{n=0}^{\infty}  p_{\triangle_D}(n)q^n  &=&       \sum_{m=2}^{\infty}    \sum_{\left(\begin{array}{c}\lambda_1>\cdots >\lambda_m>0\\ \lambda_1=  \lambda_2 + \lambda_m \end{array}\right)} \frac{q^{\lambda_1}}{(1-q^{\lambda_1}) }\cdots  \frac{q^{\lambda_m}}{(1-q^{\lambda_m}) }         \\
\sum_{n=0}^{\infty}  p_{M_0}(n)q^n  &=&    \sum_{m=2}^{\infty} \sum_ {\left(\lambda_1>\cdots >\lambda_m>0\right)}  \sum_{k_1>k_m>0} q^{k_1\lambda_1} \left(\prod_{i=2}^{m-1}\frac{q^{\lambda_i }}{ 1- \lambda_i} \right) q^{k_m\lambda_m}  \\
  \sum_{n=0}^{\infty}  p_{M_1}(n)q^n  &=&  \sum_{m=2}^{\infty} \sum_ {\left(\lambda_1>\cdots >\lambda_m>0\right)}  \sum_{k_m>k_1>0} q^{k_1\lambda_1} \left(\prod_{i=2}^{m-1}\frac{q^{\lambda_i }}{ 1- \lambda_i} \right) q^{k_m\lambda_m}         \\
 \end{eqnarray*}
 
 \begin{eqnarray*}
 \sum_{n=0}^{\infty}  p_{ \triangle_{00} }(n)q^n  &=&    \sum_{m=3}^{\infty} \sum_ {\left(\begin{array}{c}   \lambda_1>\cdots >\lambda_m>0\\  \lambda_1<\lambda_2+\lambda_m\\ 2\lambda_2 <\lambda_1 + \lambda_3    \end{array}      \right)}   \left(\prod_{i=1}^{m}\frac{q^{\lambda_i }}{ 1- \lambda_i} \right) \\
 && +  \sum_ {\left(\begin{array}{c}   \lambda_1>\lambda_2>0\\  \lambda_1 < 2\lambda_2, 3\lambda_2 < 2\lambda_1  \end{array}      \right)}   \left(\frac{q^{\lambda_1 }}{ 1- \lambda_1}\cdot \frac{q^{\lambda_2 }}{ 1- \lambda_2} \right) \\
  \sum_{n=0}^{\infty}  p_{ \triangle_{01} }(n)q^n  &=&    \sum_{m=3}^{\infty} \sum_ {\left(\begin{array}{c}   \lambda_1>\cdots >\lambda_m>0\\  \lambda_1<\lambda_2+\lambda_m\\ 2\lambda_2 >\lambda_1 + \lambda_3    \end{array}      \right)}   \left(\prod_{i=1}^{m}\frac{q^{\lambda_i }}{ 1- \lambda_i} \right) \\
 && +  \sum_ {\left(\begin{array}{c}   \lambda_1>\lambda_2>0\\  \lambda_1 < 2\lambda_2, 3\lambda_2 > 2\lambda_1  \end{array}      \right)}   \left(\frac{q^{\lambda_1 }}{ 1- \lambda_1}\cdot \frac{q^{\lambda_2 }}{ 1- \lambda_2} \right) \\
  \sum_{n=0}^{\infty}  p_{ \triangle_{10} }(n)q^n  &=&    \sum_{m=3}^{\infty} \sum_ {\left(\begin{array}{c}   \lambda_1>\cdots >\lambda_m>0\\   \lambda_1>\lambda_2+\lambda_m\\ 
   \lambda_1 <\lambda_2 + 2\lambda_m    \end{array}      \right)}   \left(\prod_{i=1}^{m}\frac{q^{\lambda_i }}{ 1- \lambda_i} \right) \\
 && +  \sum_ {\left(\begin{array}{c}   \lambda_1>\lambda_2>0\\ 3\lambda_2> \lambda_1 >2\lambda_2 \end{array}      \right)}   \left(\frac{q^{\lambda_1 }}{ 1- \lambda_1}\cdot \frac{q^{\lambda_2 }}{ 1- \lambda_2} \right) \\
  \sum_{n=0}^{\infty}  p_{ \triangle_{11} }(n)q^n  &=&    \sum_{m=3}^{\infty} \sum_ {\left(\begin{array}{c}   \lambda_1>\cdots >\lambda_m>0\\   \lambda_1>\lambda_2+2\lambda_m   \end{array}      \right)}   \left(\prod_{i=1}^{m}\frac{q^{\lambda_i }}{ 1- \lambda_i} \right) \\
 && +  \sum_ {\left(\begin{array}{c}   \lambda_1>\lambda_2>0\\  \lambda_1 >3\lambda_2 \end{array}      \right)}   \left(\frac{q^{\lambda_1 }}{ 1- \lambda_1}\cdot \frac{q^{\lambda_2 }}{ 1- \lambda_2} \right)
  \end{eqnarray*}\\

\begin{eqnarray*}
 \sum_{n=0}^{\infty}  p_{T_0 \triangle_{00} }(n)q^n  &=&    \sum_{m=2}^{\infty} \sum_ {\left(\begin{array}{c}   \lambda_1>\cdots >\lambda_m>0\\  \lambda_1<\lambda_2+\lambda_m  \end{array}      \right)}   \sum_{k_1=2}^{\infty} \sum_{\left( \begin{array}{c}k_i>0, i=1, \ldots m\\  k_m<k_1  \end{array}   \right)   }      q^{k_1 \lambda_1}       \left(\prod_{i=2}^{m-1}\frac{q^{\lambda_i }}{ 1- \lambda_i} \right) q^{k_m \lambda_m}\\
   \sum_{n=0}^{\infty}  p_{ T_0\triangle_{01} }(n)q^n  &=&   \sum_{m=2}^{\infty} \sum_ {\left(\begin{array}{c}   \lambda_1>\cdots >\lambda_m>0\\  \lambda_1>\lambda_2+\lambda_m  \end{array}      \right)}   \sum_{k_1=2}^{\infty} \sum_{\left( \begin{array}{c}k_i>0, i=1, \ldots m\\  k_m<k_1  \end{array}   \right)   }      q^{k_1 \lambda_1}       \left(\prod_{i=2}^{m-1}\frac{q^{\lambda_i }}{ 1- \lambda_i} \right) q^{k_m \lambda_m}\\
    \\
   \sum_{n=0}^{\infty}  p_{ T_1\triangle_{10} }(n)q^n  &=&     \sum_{m=2}^{\infty} \sum_ {\left(\begin{array}{c}   \lambda_1>\cdots >\lambda_m>0\\  \lambda_1<\lambda_2+\lambda_m  \end{array}      \right)}   \sum_{k_m=2}^{\infty} \sum_{\left( \begin{array}{c}k_i>0, i=1, \ldots m\\  k_1<k_m  \end{array}   \right)   }      q^{k_1 \lambda_1}       \left(\prod_{i=2}^{m-1}\frac{q^{\lambda_i }}{ 1- \lambda_i} \right) q^{k_m \lambda_m}\\
        \\
      \sum_{n=0}^{\infty}  p_{ T_1\triangle_{11} }(n)q^n  &=&    \sum_{m=2}^{\infty} \sum_ {\left(\begin{array}{c}   \lambda_1>\cdots >\lambda_m>0\\  \lambda_1>\lambda_2+\lambda_m  \end{array}      \right)}   \sum_{k_m=2}^{\infty} \sum_{\left( \begin{array}{c}k_i>0, i=1, \ldots m\\  k_1<k_m  \end{array}   \right)   }      q^{k_1 \lambda_1}       \left(\prod_{i=2}^{m-1}\frac{q^{\lambda_i }}{ 1- \lambda_i} \right) q^{k_m \lambda_m}\\
       \end{eqnarray*}\\

\begin{eqnarray*}
 \sum_{n=0}^{\infty}  p_{ T_0(T_0( \triangle_{00} )) }(n)q^n  &=&  \sum_{m=3}^{\infty} \sum_{\lambda_1>\cdots >\lambda_m>0} \sum_{\left(  \begin{array}{c} k_i>0, i=1, m-1, m \\ k_{m-1}<k_m<k_1 \end{array}   \right)  } 
 q^{k_1\lambda_1} \prod_{i=1}^{m-2} \frac{q^{     \lambda_i}}{1-q^{\lambda_i}}    q^{k_{m-1}\lambda_{m-1}}q^{k_m   \lambda_m}  \\
 &&+   \sum_{\lambda_1 >\lambda_2>0} \sum_{\left(  \begin{array}{c} k_1, k_2>0, \\ k_{2}<k_1<2k_2 \end{array}   \right)  } 
 q^{k_1\lambda_1} q^{k_2\lambda_2}  \\
  \sum_{n=0}^{\infty}  p_{  T_1(T_0( \triangle_{01} ))    }(n)q^n  &=&       \sum_{m=3}^{\infty} \sum_{\lambda_1>\cdots >\lambda_m>0} \sum_{\left(  \begin{array}{c} k_1, k_m>0 \\ k_{1}<k_m<2k_1 \end{array}   \right)  } 
 q^{k_1\lambda_1} \prod_{i=1}^{m-1} \frac{q^{     \lambda_i}}{1-q^{\lambda_i}}    q^{k_m   \lambda_m}  \\
  &&+   \sum_{\lambda_1 >\lambda_2>0} \sum_{\left(  \begin{array}{c} k_1, k_2>0, \\ k_{1}<k_2<2k_1 \end{array}   \right)  } 
 q^{k_1\lambda_1} q^{k_2\lambda_2}  \\
    \sum_{n=0}^{\infty}  p_{    T_0(T_1( \triangle_{10} ))}(n)q^n  &=&         \sum_{m=3}^{\infty} \sum_{\lambda_1>\cdots >\lambda_m>0} \sum_{\left(  \begin{array}{c} k_1, k_{m-1}, k_m>0 \\ k_{m}<k_1, k_m<k_{m-1} \end{array}   \right)  } 
 q^{k_1\lambda_1} \prod_{i=1}^{m-2} \frac{q^{     \lambda_i}}{1-q^{\lambda_i}} q^{k_{m-1}   \lambda_{m-1}}    q^{k_m   \lambda_m}  \\
   &&    +   \sum_{\lambda_1 >\lambda_2>0} \sum_{\left(  \begin{array}{c} k_1, k_2>0, \\ 2k_2<k_1 \end{array}   \right)  } 
 q^{k_1\lambda_1} q^{k_2\lambda_2}  \\   
  \sum_{n=0}^{\infty}  p_{ T_1(T_1( \triangle_{11} )) }(n)q^n  &=&         \sum_{m=3}^{\infty} \sum_{\lambda_1>\cdots >\lambda_m>0} \sum_{\left(  \begin{array}{c} k_1, k_m>0 \\  2k_1<k_{m} \end{array}   \right)  } 
 q^{k_1\lambda_1} \prod_{i=1}^{m-1} \frac{q^{     \lambda_i}}{1-q^{\lambda_i}}   q^{k_m   \lambda_m}  \\     
   \end{eqnarray*}\\

\begin{eqnarray*}
\sum_{n=0}^{\infty} p_D(n) q^n &=&  \prod_{k=1}^{\infty}  (1 + q^k)\\
\sum_{n=2}^{\infty}  p_{ \mathcal{E}_0} (n) q^n &=& \sum_{m=2}^{\infty} (1+q)\cdots (1+q^{m-1}) q^{2m}\\
\sum_{n=2}^{\infty}  p_{ \mathcal{E}_1} (n) q^n &=& \sum_{k=1}^{\infty} q^{2k} \prod_{n>k}(1+q^{m})\\
\sum_{n=2}^{\infty}  p_{ \mathcal{E}_1} (n) q^n &=& \sum_{1\leq k_1<k_m} q^{2k_1} \prod_{n=k_1+1}^{k_m-1} (1+q^{n})q^{2k_m}\\
\sum_{n=0}^{\infty} p_{\mathcal{O}}(n) q^n &=& \prod_{k=0}^{\infty} \frac{1}{1-q^{2k+1}}\\
\sum_{n=2}^{\infty} p_{ \mathcal{F}_0}(n)q^n &=&\sum_{m=2}^{\infty} \sum_{k_m=1}^{\infty} \sum_{\left(  \begin{array}{c}  \lambda_1>\cdots >\lambda_m>0\\ \lambda_m \;\mbox{even} \\  \lambda_i \; \mbox{odd}\; \mbox{if}\; i<m\\ k_1>k_m\end{array} \right)}         q^{k_1\lambda_1}\left( \prod_{i=2}^{m-1} \frac{ q^{\lambda_i}}{1-q^{\lambda_i}}  \right) q^{k_m\lambda_m}  \\
\sum_{n=2}^{\infty} p_{ \mathcal{F}_0}(n)q^n &=&      \sum_{m=2}^{\infty} \sum_{k_1=1}^{\infty} \sum_{\left(  \begin{array}{c}  \lambda_1>\cdots >\lambda_m>0\\ \lambda_1 \;\mbox{even} \\  \lambda_i \; \mbox{odd}\; \mbox{if}\; i>1\\ k_1<k_m\end{array} \right)}           q^{k_1\lambda_1}\left( \prod_{i=2}^{m-1} \frac{ q^{\lambda_i}}{1-q^{\lambda_i}}  \right) q^{k_m\lambda_m}  
\end{eqnarray*}

Then the generating interpretation of  Theorem \ref{T0 and T1 1-1} and Theorem \ref{theorem 3.3 neo}  is that 

$$\sum_{n=2}^{\infty}  p_{\triangle_0} (n) q^n = \sum_{n=2}^{\infty}  p_{ \mathcal{M}_0} (n) q^n, \; \sum_{n=2}^{\infty}  p_{\triangle_1} (n) q^n = \sum_{n=2}^{\infty}  p_{ \mathcal{M}_1} (n) q^n.$$

 The generating function interpretation for Theorems \ref{firstcylinder} and   \ref{secondcylinder} is
 $$\begin{array}{ccccc}
 \sum_{n=2}^{\infty}    p_{\triangle_{00}}(n) q^n &=&  \sum_{n=2}^{\infty}    p_{T_0(\triangle_{00})}(n) q^n&=&  \sum_{n=2}^{\infty}     p_{T_0(T_0(\triangle_{00}))} (n) q^n\\
  \sum_{n=2}^{\infty}    p_{\triangle_{01}}(n) q^n &=&  \sum_{n=2}^{\infty}    p_{T_0(\triangle_{01})}(n) q^n&=&  \sum_{n=2}^{\infty}     p_{T_1(T_0(\triangle_{01}))}   (n) q^n \\
   \sum_{n=2}^{\infty}    p_{\triangle_{10}}(n) q^n &=&  \sum_{n=2}^{\infty}    p_{T_1(\triangle_{10})}(n) q^n&=&  \sum_{n=2}^{\infty}     p_{T_0(T_1(\triangle_{10}))}   (n) q^n \\
    \sum_{n=2}^{\infty}    p_{\triangle_{11}}(n) q^n &=&  \sum_{n=2}^{\infty}    p_{T_1(\triangle_{11})}(n) q^n&=&  \sum_{n=2}^{\infty}     p_{T_1(T_1(\triangle_{11}))}   (n) q^n.
 \end{array}$$
 
 Theorem \ref{distinct}  is now
 
  $$\sum_{n=0}^{\infty} p_D(n) q^n  = \sum_{n=0}^{\infty} q^n  +\sum_{n=2}^{\infty}  p_{ \mathcal{E}_0} (n) q^n +\sum_{n=2}^{\infty}  p_{ \mathcal{E}_1} (n) q^n +\sum_{n=2}^{\infty}  p_{ \mathcal{E}_D} (n) q^n + \sum_{k=1}^{\infty} q^{3k}.$$
   The $\sum_{n=0}^{\infty} q^n  $ term is reflecting all partitions of the form $(n)\times [1]$ and the $ \sum_{k=1}^{\infty} q^{3k}$ is capturing all partitions of the form $(k)\times [3].$

  Now for the generating function version of Theorem \ref{odd and T}:

  $$\sum_{n=1}^{\infty} p_{\mathcal{O}}(n) q^n = \sum_{k=0}^{\infty} \frac{ q^{2k+1}}{1-q^{2k+1}}  + \sum_{n=2}^{\infty} p_{ \mathcal{F}_0}(n)q^n + \sum_{n=2}^{\infty} p_{ \mathcal{F}_0}(n)q^n.$$
  The first term on the right is capturing all the partition with odd parts of dimension one, and hence the partitions fo the form $(2n+1) \times [k].$

\section{Conclusion}

While we find the new partition identities interesting, we find the method by which they were discovered as more important, namely the recognition that the triangle map provides a map from partitions of dimension two or greater to partitions. 

The space of partitions $\mathcal{P}_{\geq 2}$ can be naturally split into three disjoint subsets $\triangle_0$, $\triangle_1$ and $\triangle_D$.  The map $T$ is one-to-one on $\triangle_0$ and $\triangle_1$.  Hence for any subset $S$ of 
$\triangle_0$ and $\triangle_1,$  we have 
$p_S(n) =p_T(S) (n)$.  This provided a new systematic method for producing many new partition identities, namely find an ``interesting'' subset $S$ in the space of partitions.  Then study $T(S), T^2(S), T^3(S), \ldots $ and $T^{-1}(S), T^{-2}(S), T^{-3}(S), \ldots.$  The proofs of these new identities will be easy and straightforward, as we saw in the examples given in Section \ref{examples}.

There are, though,  many other multi-dimensional continued fraction algorithms.  Most seem not to be useful for partitions, as briefly discussed in section six of \cite{BDGS} for partitions in $\mathcal{P}_{3}$. 
We do not really understand why the triangle map, and a few other multi-dimensional continued fraction algorithms, can be used on partitions while most cannot.  Is there an underlying geometric or dynamical reason for the triangle map to be, let us say, {\it partition friendly}, or is it simply a coincidence.  This strikes us  a hard (if not a mathematical precise) question.

\section{Appendix}
From Mathematica, we know that all $56$ partitions of $11$ are:

$$\begin{array}{cccc}
(11)   \times [1]  & (10, 1)   \times [1, 1]  &  (9, 2)   \times [1, 1] &  (9, 1)   \times [1, 2]\\
(8,3)   \times [1,1]  & (8, 2,1)   \times [1, 1,1]  &  (8, 1)   \times [1, 3] &  (7, 4)   \times [1, 1]\\
(7,3,1)   \times [1,1,1]  & (7, 2)   \times [1, 2]  &  (7,2,1)   \times [1, 1,2] &  (7, 1)   \times [1, 4]\\
(6,5)   \times [1,1]  & (6,4,1)   \times [1, 1,1]  &  (6, 3,2)   \times [1, 1,1] &  (6,3,1)   \times [1, 1,2]\\
(6,2,1)   \times [1,2,1]  & (6,2,1)   \times [1, 1,3]  &  (6, 1)   \times [1, 5] &  (5,1)   \times [2, 1]\\
(5,4,2)   \times [1,1,1]  & (5,4,1)   \times [1,1,2]  &  (5,3)   \times [1, 2] &  (5,3,2,1)   \times [1, 1,1,1]\\
  (5,3,1)   \times [1, 1,3] & (5,2)   \times [1, 3]  &    (5,2,1)   \times [1, 2,2]      & (5,2,1)   \times [1, 1,4]    \\
     (5,1)   \times [1, 6]      &    (4,3)   \times [2,1]    &          (4,2,1)   \times [2,1,1]   & (4,1)   \times [2,3]        \\
         (4,3,1)   \times [1,2,1]      &    (4,3,2)   \times [1,1,2]      &    (4,3,2,1)   \times [1,1,1,2]   & (4,3,1)   \times [1,1,4]         \\
  (4,2,1)   \times [1,3,1]     &   (4,2,1)   \times [1,2,3]       &    (4,2,1)   \times [1,1,5]    & (4,1)   \times [1,7]      \\
       (3,2)   \times [3,1]            &    (3,1)   \times [3,2]      &     (3,2,1)   \times [2,2,1]     &    (3,2,1)   \times [2,1,3]       \\
      (3,1)   \times [2,5]                             &   (3,2)   \times [1,4]       &     (3,2,1)   \times [1,3,2]     &    (3,2,1)   \times [1,2,4]       \\
            (3,2,1)   \times [1,1,6]                               &    (3,1)   \times [1,8]      &    (2,1)   \times [5,1]      &    (2,1)   \times [4,3]       \\
               (2,1)   \times [3,5]                                    &     (2,1)   \times [2,7]     &    (2,1)   \times [1,9]      &    (1) \times [11]    \\
\end{array}$$

We encourage readers to use this list to check all the partition identities that are given in this paper.

\end{document}